\documentclass[reqno,12pt]{amsart}
\usepackage{mathtools}
\usepackage{amsfonts}
\usepackage{amsmath}
\usepackage{amsthm}
\usepackage{amssymb}
\usepackage{txfonts}
\usepackage{enumerate}
\usepackage[numbers, sort&compress]{natbib}
\numberwithin{equation}{section} \theoremstyle{plain}
\newtheorem{theorem}{Theorem}[section]
\newtheorem{coro}[theorem]{Corollary}
\newtheorem{lemma1235}[theorem]{Lemma}
\theoremstyle{definition}

\newtheorem*{re}{Remark}
\theoremstyle{plain}
\newtheorem{theoremalph}{theorem}

\newtheorem{atm}[theoremalph]{Theorem}

\def\begr{\begin{eqnarray}} \def\endr{\end{eqnarray}}

\begin{document}
\title
{NONLINEAR DIFFERENTIAL EQUATION AND ANALYTIC FUNCTION SPACES}
\author{Hao Li \quad and \quad Songxiao Li$^*$ }
\address{College of Mathematics and Information Science, Henan Normal University, Xinxiang 453007, China}\email{lihao20102010@163.com(H. Li)}
\vskip 3mm
\address{Institute of Systems Engineering, Macau University of Science and Technology,  Avenida Wai Long,
Taipa, Macau. }\email{jyulsx@163.com (S. Li)} \noindent\subjclass[2000]{34M10,
30D35 }
\begin{abstract} In this paper we consider the nonlinear complex differential equation $$(f^{(k)})^{n_{k}}+A_{k-1}(z)(f^{(k-1)})^{n_{k-1}}+\cdot\cdot\cdot+A_{1}(z)(f')^{n_{1}}+A_{0}(z)f^{n_{0}}=0,
$$where $ A_{j}(z)$,  $ j=0, \cdots, k-1 $,  are analytic in the unit disk $ \mathbb{D} $, $ n_{j}\in R^{+} $
for all $ j=0, \cdots, k $. We investigate this nonlinear
differential equation from two aspects. On one hand, we provide some
sufficient conditions on coefficients such that all solutions of
this equation belong to a class of M\"{o}bius invariant function
space, the so-called  $Q_K$ space. On the other hand, we find some growth
estimates for the analytic solutions of this equation if the
coefficients belong to some analytic function spaces.
\thanks{\noindent{\it Keywords}: Nonlinear differential equation, growth estimate, Bloch space, $Q_K$ space.\\
\noindent~~~ *Corresponding author.}
\end{abstract}
\maketitle
\section{Introduction}
  Let $ \mathbb{D} $ be the unit disk in the complex plane $\mathbb{C}$. Let $ g(a,z)=-\log|\varphi_{a}(z)| $ be the Green function on $
\mathbb{D} $ with logarithmic singularity at $ a\in\mathbb{D} $,
where $\varphi_{a}(z)=(a-z)/(1-\bar{a}z) $ is the M\"{o}bius
transformation of $ \mathbb{D} $. Denote by $ H(\mathbb{D}) $ the
set of all analytic functions on $ \mathbb{D} $.  For $ s>0 $,  the
Bloch-type space, denoted by $ \mathfrak{B}^{s} $, consists of all
$f \in H(\mathbb{D})$ for which \begin{align}\label{85}
\|f\|_{\mathfrak{B}^{s}}=|f(0)|+\sup_{z\in\mathbb{D}}|f'(z)|(1-|z|^{2})^{s}<\infty.
\end{align}
 When $s=1$, we get the classical Bloch space $\mathfrak{B}$.  Obviously, $\mathfrak{B}^{s_{1}}\subsetneqq\mathfrak{B}\subsetneqq
\mathfrak{B}^{s_{2}} $ for $ 0<s_{1}<1<s_{2}<\infty $.  See
\cite{zhu2} for the theory of Bloch-type spaces. \par

 For $ 0\leqslant s<\infty$, $ 0<t<\infty $, the weighted Hardy space, denoted by  $ H_{s}^{t} $,  consists
of all $ f\in H(\mathbb{D}) $ satisfying
\begin{align} &\|f\|_{H_{s}^{t}}=\sup_{0\leqslant r<1}(1-r^{2})^{s}\left(\frac{1}{2\pi}\int_{0}^{2\pi}
|f(re^{i\varphi})|^{t}d\varphi\right)^{\frac{1}{t}}<\infty .
\end{align}
 The Bers-type space, denoted by $H^\infty_s$, consists of all $f\in H(\mathbb{D})$  such that
  \begin{align} \|f\|_{ H^\infty_s}=\sup_{z\in \mathbb{D}} (1-|z|^2)^s \left|  f(z)\right|<\infty.\label{75}
\end{align}
It is easy to check that $H^\infty_s$ is a Banach space with the
norm $\|\cdot \|_{ H^\infty_s}$. The Proposition 7 in \cite{zhu2}
gives that $ \mathfrak{B}^{s+1}=H_{s}^{\infty}$ for $ s>0 $.

For a nondecreasing function $ K:[0,\infty)\rightarrow[0,\infty) $, the $ Q_{K} $ space consists of all $ f\in H(\mathbb{D}) $  for
which \begin{align}\label{21}
\|f\|^{2}_{Q_{K}}=\sup_{a\in\mathbb{D}}\int_{\mathbb{D}}|f'(z)|^{2}K(g(a,z))d\sigma(z)<\infty.
\end{align}
Here  $ d\sigma $ is an area measure on $ \mathbb{D} $ normalized
such that $ \sigma(\mathbb{D})= 1 $. When $K(t)=t^p$, then $
Q_{K} $ space coincides with the $Q_p$ space. For all $ p>1 $, the $
Q_{p} $ space is  equal to the Bloch space $ \mathfrak{B} $. The authors
in \cite{ww} proved that $ Q_{K}\subseteq \mathfrak{B} $, and hence
we obtain that $ Q_{K}\subsetneq H_{\alpha}^{\infty} $ for any
$\alpha\in (0,\infty) $.\par



 The growth of  solutions of the linear differential equation
\begin{align}
f^{(k)}+A_{k-1}(z)f^{(k-1)}+\cdot\cdot\cdot+A_{1}(z)f'+A_{0}(z)f=0
\label{25}\end{align}
  with analytic coefficients has attracted a lot
of attention, see for example, \cite{be,cy,cs,cgh,gsw,
he,hkr1,hkr2,hkr3,hkr4,hi,kr,kr1,la,hw, po,r}. Ch. Pommerenke \cite{po} studied the second-order equation
\begin{align} f''+A(z)f=0, \label{2000} \end{align} where $ A(z) $ is an
analytic function in $ \mathbb{D} $. He found some sufficient
conditions on the coefficient function $ A(z) $ such that all
solutions of \eqref{2000} are in $ H^{2} $.   Heittokangas  in
\cite{he} found a sufficient condition on the coefficient function $
A(z) $ such that all solutions of \eqref{2000} are in $
\cap_{0<p<\infty}Q_{p}$.  In 2011,   Li and  Wulan in \cite{hw}
studied the equation \eqref{25} and gave the sufficient conditions
on coefficients of \eqref{25} such that all solutions of \eqref{25}
belong to the $Q_{K} $ space.  The main results are stated as
follows.

\begin{atm}[{\cite[Theorem 2.1]{hw}}]  \label{32} Let $ 1<c<3/2 $ and let $ K $ satisfy \begin{align}
  \int_{1}^{\infty}\frac{\varphi_{K}(s)}{s^{2c-1}}ds<\infty,
\label{22}
\end{align} where $ \varphi_{K}(s)=\sup_{0\leqslant t \leqslant 1}K(st)/K(t),\,\, 0<s<\infty. $
There exists a constant $ \alpha=\alpha(k, c, K)>0 $ such that if
the coefficients of \eqref{25} satisfy the following
\begin{align} &\|A_{j}\|_{H_{k-j}^{\infty}}=\sup_{z\in\mathbb{D}}|A_{j}(z)|(1-|z|^{2})^{k-j}
\leqslant \alpha,\,j=1,\cdot\cdot\cdot,k-1,\notag
\shortintertext{and}
&\|A_{0}\|_{H_{k-c}^{\infty}}=\sup_{z\in\mathbb{D}}|A_{0}(z)|(1-|z|^{2})^{k-c}\leqslant
\alpha, \notag \end{align} then all solutions of \eqref{25} belong
to the $ Q_{K} $ space.
\end{atm}

\begin{atm}[{\cite[Theorem 2.6]{hw}}]\label{47} Let the nondecreasing function $ K $ satisfy
\begin{align}\int_{0}^{1}\frac{\varphi_{K}(s)}{s}ds<\infty. \label{43} \end{align}
 There exists a constant $ \beta=\beta(k, K)>0 $ such that if the
coefficients of \eqref{25} satisfy
 \begin{align}
&\|A_{j}\|_{H_{k-j}^{\infty}}=\sup_{z\in\mathbb{D}}|A_{j}(z)|(1-|z|^{2})^{k-j}
\leqslant \beta,\,\,j=1,\cdot\cdot\cdot,k-1,\notag
\shortintertext{and}
&\|A_{0}\|_{H_{k-1}^{\infty}}=\sup_{z\in\mathbb{D}}|A_{0}(z)|(1-|z|^{2})^{k-1}\leqslant
\beta, \notag\end{align}
  then all solutions of \eqref{25} belong to the $ Q_{K} $ space.
\end{atm}

In 2004, Heittokangas, Korhonen and R\"{a}tty\"{a} in \cite{hkr1}
studied the equation \eqref{25}, and gave the estimate of solutions
of \eqref{25} by using the Herold's comparison theorem (see
\cite{her}).\par

\begin{atm}[{\cite[Theorem 5.1]{hkr1}}] \label{73} Let $ f $  be a solution of \eqref{25} in the disc $ D_{R}=\{z\in
\mathbb{C}: |z|<R\} $, $ 0<R<\infty $, let $ n_{c} \in \{1,
\cdot\cdot\cdot, k\} $ be the number of nonzero coefficients $
A_{j}(z) $, $ j=0, 1, 2, \cdots, k-1,$ and let $ \theta\in [0, 2\pi)
$ and $ \varepsilon>0 $. If $ z_{\theta}=\nu e^{i\theta}\in D_{R} $
is such that $ A_{j}(z_{\theta})\neq 0 $ for some $ j=0,\cdots, k-1
$, then, for all $ \nu<r<R $,
\begin{align}|f(re^{i\theta})|\leqslant C\exp
\bigg(n_{c}\int_{\nu}^{r}\max_{j=0,\cdots,
k-1}|A_{j}(te^{i\theta})|^{1/(k-j)}dt\bigg),\end{align} where $ C>0
$ is a constant satisfying
\begin{align} C\leqslant (1+\varepsilon)\max_{j=0,\cdots, k-1}\bigg(\frac{|f^{(j)}(z_{\theta})|}{(n_{c})^{j}
\max_{n=0,\cdots, k-1}|A_{n}(z_{\theta})|^{j/(k-n)}}\bigg).
\end{align}
\end{atm}

In this paper, we consider the nonlinear differential equation
\begin{align}\label{19}
(f^{(k)})^{n_{k}}+A_{k-1}(z)(f^{(k-1)})^{n_{k-1}}+\cdots+A_{1}(z)(f')^{n_{1}}+A_{0}(z)(f)^{n_{0}}=0,
\end{align}
where coefficients are analytic in $\mathbb{D} $, and $ n_{j}\in
R^{+} $ for all $ j=0, 1, \cdots, k $. We study it from two aspects.
On one hand, we find sufficient conditions on the coefficients of
the differential equation such that all solutions of \eqref{19}
belong to  the $ Q_{K} $ space. On the other hand, we find growth
estimate for solutions of \eqref{19}, if the  coefficients  belong
to some analytic function spaces. More precisely, in section
\ref{71}, the main results generalize Theorems \ref{32} and \ref{47}
to the case of nonlinear differential equation \eqref{19}. In
section \ref{72}, we generalize   Theorem \ref{73} to the case of
the nonlinear differential equation and give some estimates for
solutions of \eqref{19} by the coefficients, and then we  give the
growth estimate of solutions as the coefficients belong to some
given function spaces, which can be seen as the inverse of the
results in Section \ref{71}.

In this paper, the letter $ C $ denotes a positive constant through
the paper which may vary at each occurrence.

\section{Constraints on coefficients}\label{71}
In this section, we generalize Theorems \ref{32} and \ref{47} to the
case of nonlinear differential equation \eqref{19}. For this
purpose, we need  the following result ({see \cite{wuzhu}}).\par

\begin{lemma1235}\label{39} Suppose $ K $ satisfies \eqref{22} or
\eqref{43}. Then for any positive integer  $ k $ and any $ f \in H(
\mathbb{D}) $ we have $f \in Q_{K} $ if and only if
\begin{align}\label{38} \sup_{a\in\mathbb{D}}\int_{\mathbb{D}}|f^{(k)}(z)|^{2}(1-|z|^{2})^{2k-2}
K(1-|\varphi_{a}(z)|^{2}) d\sigma(z)<\infty.
\end{align}
\end{lemma1235}

\begin{re}
Note that by the proofs in \cite{wuzhu},  the $ Q_{K} $ norm and the
area integral in Lemma \ref{39} are comparable; that is, there
exists a constant $ C > 0 $, independent of $ f $, such that
\begin{align} \label{40}
C^{-1}\|f\|_{Q_{K}}^{2}\leqslant \sup_{a\in
\mathbb{D}}\int_{\mathbb{D}}|f^{(k)}(z)|^{2}(1-|z|^{2})^{2k-2}K(1-|\varphi_{a}(z)|^{2})d\sigma(z)\leqslant
C\|f\|_{Q_{K}}^{2}.
\end{align}
\end{re}

\begin{theorem}\label{20} Suppose  $ n_{k}\geqslant n_{j}>1 $ for all $ j=0, 1, \cdots, k-1 $.
Let $ 1<c<3/2 $ and let $ K $ satisfy \eqref{22}. There exists a
positive constant $ \alpha $ depending only on $ k $, $ c $, $ n_{k}
$ and $ K $, such that if the coefficients of \eqref{19} are
analytic in $ \mathbb{D} $ and satisfy the following
\begin{align} &\|A_{j}\|_{H_{k-j}^{\infty}}=\sup_{z\in\mathbb{D}}|A_{j}(z)|(1-|z|^{2})^{n_{k}(k-j)}
\leqslant \alpha,\,j=1,\cdot\cdot\cdot,k-1, \label{41}
\shortintertext{and}
&\|A_{0}\|_{H_{k-c}^{\infty}}=\sup_{z\in\mathbb{D}}|A_{0}(z)|(1-|z|^{2})^{n_{k}(k-c)}\leqslant
\alpha, \label{42} \end{align} then all solutions of \eqref{19}
belong to the  $ Q_{K} $ space.
\end{theorem}

\begin{proof} Let $ f $ be a non-trivial solution of \eqref{19}, and let $ f_{\rho}(z)=f(\rho z) $ and
$ A_{j, \rho}(z)=A_{j}(\rho z) $, where $ 1/2\leqslant \rho < 1 $.
Then
  \begr\label{52}
&&\int_{\mathbb{D}}|f_{\rho}^{(k)}(z)|^{2} (1-|z|^{2})^{2k-2}K(1-|\varphi_{a}(z)|^{2}) d\sigma(z)\notag \\
&=&\int_{\mathbb{D}}\bigg|\sum_{j=0}^{k-1}A_{j, \rho}
(z)\rho^{kn_{k}-jn_{j}}\big(f_{\rho}^{(j)}(z)\big)^{n_{j}}\bigg|^{2/n_{k}}(1-|z|^{2})^{2k-2}K(1-|\varphi_{a}(z)|^{2})
d\sigma(z)\notag\\&\leqslant&
k^{2/n_{k}}\int_{\mathbb{D}}\sum_{j=0}^{k-1}|A_{j,
\rho}(z)|^{2/n_{k}}|f_{\rho}^{(j)}(z)|^{2n_{j}/n_{k}}
(1-|z|^{2})^{2k-2}K(1-|\varphi_{a}(z)|^{2}) d\sigma(z)\notag
\endr
\begr &\leqslant&
k^{2/n_{k}}\alpha^{2/n_{k}}\sum_{j=1}^{k-1}\int_{\mathbb{D}}|f_{\rho}^{(j)}(z)|^{2n_{j}/n_{k}}
(1-|z|^{2})^{2j-2}K(1-|\varphi_{a}(z)|^{2}) d\sigma(z)\notag\\
&\quad&
+k^{2/n_{k}}\alpha^{2/n_{k}}\int_{\mathbb{D}}|f_{\rho}(z)|^{2n_{0}/n_{k}}
(1-|z|^{2})^{2c-2}K(1-|\varphi_{a}(z)|^{2}) d\sigma(z)\notag\\&:=&
k^{2/n_{k}}\alpha^{2/n_{k}}(J_{1}+J_{2}).
\endr
For $ J_{1} $, by H\"{o}lder inequality with the measure $(1-|z|^{2})^{2j-2}K(1-|\varphi_{a}(z)|^{2}) d\sigma(z) $, we have
\begin{align*}
J_{1}&=\sum_{j=1}^{k-1}\int_{\mathbb{D}}|f_{\rho}^{(j)}(z)|^{2n_{j}/n_{k}}(1-|z|^{2})^{2j-2}K(1-|\varphi_{a}(z)|^{2}) d\sigma(z)\\
&\leqslant
C\sum_{j=1}^{k-1}\bigg(\int_{\mathbb{D}}|f_{\rho}^{(j)}(z)|^{2}
(1-|z|^{2})^{2j-2}K(1-|\varphi_{a}(z)|^{2})
d\sigma(z)\bigg)^{n_{j}/n_{k}}.
\end{align*}
By Lemma \ref{39} we have $J_{1} \leqslant C\|f_{\rho}\|^{2n_{j}/n_{k}}_{Q_{K}}.$ Without loss of generality, we can assume that $\|f_{\rho}\|_{Q_{K}}>1, $ and it follows that
\begin{align}
J_{1} \leqslant C\|f_{\rho}\|^{2}_{Q_{K}}.
\end{align}
For $ J_{2} $, by H\"{o}lder inequality with measure $(1-|z|^{2})^{2c-2}K(1-|\varphi_{a}(z)|^{2}) d\sigma(z) $, we have
\begin{align*}
J_{2}&=\int_{\mathbb{D}}|f_{\rho}(z)|^{2n_{0}/n_{k}}(1-|z|^{2})^{2c-2}K(1-|\varphi_{a}(z)|^{2}) d\sigma(z)\\
&\leqslant
C\bigg(\int_{\mathbb{D}}|f_{\rho}(z)|^{2}(1-|z|^{2})^{2c-2}K(1-|\varphi_{a}(z)|^{2})
d\sigma(z)\bigg)^{n_{0}/n_{k}}.
\end{align*}
Then using the reasoning used in \cite[Theorem 2.1]{hw} we have
\begr
J_{2} \leqslant  C(\|f\|^{2}_{Q_{K}}+|f(0)|^{2})^{n_{0}/n_{k}} \leqslant C(\|f_{\rho}\|^{2n_{0}/n_{k}}_{Q_{K}}+|f(0)|^{2n_{0}/n_{k}}).\nonumber
\endr
Without loss of generality, we can assume that $\|f_{\rho}\|_{Q_{K}}>1, $ for otherwise the conclusion is clearly
established. It follows that
 \begr
 J_{2} \leqslant C\|f_{\rho}\|^{2}_{Q_{K}}+C|f(0)|^{2n_{0}/n_{k}}.
\endr
Combining \eqref{52} with the estimates for $ J_{1} $ and $ J_{2} $,
we have
\begin{align}
(1-k^{2/n_{k}}\alpha^{2/n_{k}}C)\|f_{\rho}\|^{2}_{Q_{K}}\leqslant
k^{2/n_{k}}\alpha^{2/n_{k}}C|f(0)|^{2n_{0}/n_{k}},
\end{align}
where the constant $ C $ depends only on $ k $, $ c $, $ n_{k} $ and
$ K $. The conclusion $ f\in Q_{K} $ follows by choosing $ \alpha $
sufficiently small and letting $ \rho\rightarrow 1 $. The proof is
complete.
\end{proof}

\begin{theorem}\label{53}
Suppose $ n_{k}\geqslant n_{j}>1 $ for all $ j=0, 1, \cdots, k-1 $.
Let the nondecreasing function $ K $ satisfy \eqref{43}. There
exists a positive constant $ \beta $ depending only on $ k $, $n_{k}
$ and $ K $ such that if the coefficients of \eqref{19} satisfy
\begin{align}
&\|A_{j}\|_{H_{k-j}^{\infty}}=\sup_{z\in\mathbb{D}}|A_{j}(z)|(1-|z|^{2})^{n_{k}(k-j)}
\leqslant \beta,\,\,j=1,\cdot\cdot\cdot,k-1, \shortintertext{and}
&\|A_{0}\|_{H_{k-1}^{\infty}}=\sup_{z\in\mathbb{D}}|A_{0}(z)|(1-|z|^{2})^{n_{k}(k-1)}\leqslant
\beta, \end{align}
 then all solutions of \eqref{19} belong to the $Q_{K} $ space.
\end{theorem}

\begin{proof} Denote $ f_{\rho}(z)=f(\rho z) $ and $ A_{j,\rho}(z)=A_{j}(\rho z)$
for $ 1/2\leqslant \rho<1 $. By Lemma \ref{39} we have the following
  \begr
  \label{54}&&\int_{\mathbb{D}}|f_{\rho}^{(k)}(z)|^{2}(1-|z|^{2})^{2k-2}K(1-|\varphi_{a}(z)|^{2})d\sigma(z)\nonumber\\
&=&\int_{\mathbb{D}}\bigg|\sum_{j=0}^{k-1}A_{j,\rho}
(z)\rho^{kn_{k}-jn_{j}}\big(f_{\rho}^{(j)}(z)\big)^{n_{j}}\bigg|^{2/n_{k}}(1-|z|^{2})^{2k-2}K(1-|\varphi_{a}(z)|^{2})d\sigma(z)\nonumber\\
&\leqslant&
k^{2/n_{k}}\int_{\mathbb{D}}\sum_{j=0}^{k-1}|A_{j,\rho}(z)|^{2/n_{k}}|f_{\rho}^{(j)}(z)|^{2n_{j}/n_{k}}
(1-|z|^{2})^{2k-2}K(1-|\varphi_{a}(z)|^{2})d\sigma(z)\nonumber\\
 &\leqslant& k^{2/n_{k}}\beta^{2/n_{k}}\sum_{j=1}^{k-1}\int_{\mathbb{D}}|f_{\rho}^{(j)}(z)|^{2n_{j}/n_{k}}
(1-|z|^{2})^{2j-2}K(1-|\varphi_{a}(z)|^{2})
d\sigma(z)\nonumber\\
&&\quad+k^{2/n_{k}}\beta^{2/n_{k}}\int_{\mathbb{D}}|f_{\rho}(z)|^{2n_{0}/n_{k}}K(1-|\varphi_{a}(z)|^{2})
d\sigma(z)\nonumber\\
 &\coloneqq& k^{2/n_{k}}\beta^{2/n_{k}}(T_{1}+T_{2}).
  \endr
 For $ T_{1}$, by H\"{o}lder inequality with the measure $
(1-|z|^{2})^{2j-2}K(1-|\varphi_{a}(z)|^{2})d\sigma(z) $, we have
\begin{align*}T_{1}&=
\sum_{j=1}^{k-1}\int_{\mathbb{D}}|f_{\rho}^{(j)}(z)|^{2n_{j}/n_{k}}
(1-|z|^{2})^{2j-2}K(1-|\varphi_{a}(z)|^{2})d\sigma(z)\nonumber\\&\leqslant
C\sum_{j=1}^{k-1}\bigg(\int_{\mathbb{D}}|f_{\rho}^{(j)}(z)|^{2}
(1-|z|^{2})^{2j-2}K(1-|\varphi_{a}(z)|^{2})
d\sigma(z)\bigg)^{n_{j}/n_{k}}.
\end{align*}
 By Lemma \ref{39}, we have $ T_{1} \leqslant C \|f_{\rho}\|^{2n_{j}/n_{k}}_{Q_{K}}. $ 
Without loss of generality, we can assume that $\|f_{\rho}\|_{Q_{K}}>1, $ and it follows that
\begin{align}\label{55}
T_{1} \leqslant C\|f_{\rho}\|^{2}_{Q_{K}}.
\end{align}
For $ T_{2} $, by using the H\"{o}lder inequality with the measure $
K(1-|\varphi_{a}(z)|^{2}) d\sigma(z) $, we have \begr
T_{2}&=&\int_{\mathbb{D}}|f_{\rho}(z)|^{2n_{0}/n_{k}}
K(1-|\varphi_{a}(z)|^{2}) d\sigma(z)\nonumber\\
&\leqslant&
\bigg(\int_{\mathbb{D}}|f_{\rho}(z)|^{2}K(1-|\varphi_{a}(z)|^{2})
d\sigma(z)\bigg)^{n_{0}/n_{k}}.\nonumber
\endr
Then, applying the reasoning used in \cite[Theorem 2.6]{hw} to $T_{2} $, we obtain the following \begin{align}
T_{2} \leqslant
\bigg(C\|f_{\rho}\|_{Q_{K}}^{2}+2|f(0)|^{2}\bigg)^{n_{0}/n_{k}} \leqslant
C\|f_{\rho}\|_{Q_{K}}^{2n_{0}/n_{k}}+2^{n_{0}/n_{k}}|f(0)|^{2n_{0}/n_{k}}.\nonumber
\end{align}
Without loss of generality, we can assume that $
\|f_{\rho}\|_{Q_{K}}>1, $ for otherwise the conclusion is clearly
established. This yields that
\begin{align}\label{56}
T_{2} \leqslant C\|f_{\rho}\|_{Q_{K}}^{2}+2^{n_{0}/n_{k}}|f(0)|^{2n_{0}/n_{k}}.
\end{align}
Combining \eqref{54}, \eqref{55} and \eqref{56} we have
\begin{align}
(1-k^{2/n_{k}}\beta^{2/n_{k}}C)\|f_{\rho}\|^{2}_{Q_{K}}\leqslant
2^{n_{0}/n_{k}}k^{2/n_{k}}\beta^{2/n_{k}}|f(0)|^{2n_{0}/n_{k}},\nonumber
 \end{align}
where the constant $ C $ depends only on $ k $, $ n_{k} $ and  $ K$. The conclusion $ f\in Q_{K} $ follows by choosing $
\beta $ sufficiently small and letting $ \rho\rightarrow 1 $. The
proof is complete.
\end{proof}

\section{Constraints on the solutions}\label{72}
 In this section, we find the growth estimate for the solutions of \eqref{19} if the coefficients of \eqref{19} belong
to some analytic function spaces. First, we extend Herold's
comparison theorem in \cite[Satz 1]{her} or \cite[Theorem H]{hkr1}
to the case of nonlinear differential equation.

\begin{lemma1235} \label{69} Suppose $ n_{0}>1 $, $ a\in [0, 1) $. Let $ A_{j}(x) $, $ j=0,
\cdots, k-1 $ be complex valued functions defined on $ [a, 1) $. Let
$ E\subset [a, 1) $ be a finite set of points, and let $ B_{j}(x) $,
$ j=0, \cdots, k-1 $ be real valued non-negative functions such that
$ |A_{k-j}(x)|\leqslant n_{0}^{-j}B_{k-j}(x) $ for all $ x \in [a,
1)\backslash E $ and  $ j=1, 2, \cdots, k. $ Moreover, let $
|A_{j}(x)| $ be continuous for all $ x \in [a, 1)\backslash E $. If
$ v(x) $ is a solution of the differential equation
\begin{align}\label{58}
&(v^{(k)}(x))^{n_{0}}-\sum_{j=1}^{k}A_{k-j}(x)(v^{(k-j)}(x))^{n_{0}}=0,
\shortintertext{and $ u(x) $ satisfies}
&u^{(k)}(x)-\sum_{j=1}^{k}B_{k-j}(x)u^{(k-j)}(x)=0\label{59}
\end{align}
on $ [a, 1)\backslash E $, where $ |v^{(k-j)}(a)|^{n_{0}}\leqslant
u^{(k-j)}(a) $ for all $ j=1, \cdots, k $, then $$
|v^{(j)}(x)|^{n_{0}}\leqslant n_{0}^{k-j}u^{(j)}(x) $$ on $ [a,
1)\backslash E $ for all $ j=0, 1, \cdots, k $.
\end{lemma1235}

\begin{proof} Without loss of generality, for convenience this Lemma is proved
by taking $ n_{0}=2 $ and $ a=0 $ such that $
|v^{(k-j)}(0)|^{n_{0}}\leqslant u^{(k-j)}(0) $ for all $ j=1,
\cdots, k $. \par

After some calculations, we have
\begin{align}
|v^{(k-j)}(x)|\leqslant
\sum_{m=1}^{j}|v^{(k-m)}(0)|\frac{x^{j-m}}{(j-m)!}+\int_{0}^{x}\frac{(x-s)^{j-1}}{(j-1)!}|v^{(k)}(s)|ds,\label{60}
\end{align}
and so
\begin{align}
|v^{(k-j)}(x)|^{2}&\leqslant
2^{j}\sum_{m=1}^{j}|v^{(k-m)}(0)|^{2}\bigg|\frac{x^{j-m}}{(j-m)!}\bigg|^{2}+2^{j}\int_{0}^{x}\bigg|\frac{(x-s)^{j-1}}{(j-1)!}
\bigg|^{2}|v^{(k)}(s)|^{2}ds\nonumber\\
&\leqslant
2^{j}\sum_{m=1}^{j}|v^{(k-m)}(0)|^{2}\frac{x^{j-m}}{(j-m)!}+2^{j}\int_{0}^{x}\frac{(x-s)^{j-1}}{(j-1)!}
|v^{(k)}(s)|^{2}ds.\label{57}
\end{align}
By \eqref{58} we have
   \begin{align*}|v^{(k)}(x)|^{2}\leqslant
\sum_{j=1}^{k}|A_{k-j}(x)||v^{(k-j)}|^{2},\end{align*}
 which,  together with \eqref{57}, shows
    \begin{align}\label{61}
|v^{(k)}(x)|^{2} &\leqslant  \sum_{j=1}^{k}|A_{k-j}(x)||v^{(k-j)}(x)|^{2}\notag\\
& \leqslant
\sum_{j=1}^{k}2^{j}|A_{k-j}(x)|\sum_{m=1}^{j}|v^{(k-m)}(0)|^{2}\frac{x^{j-m}}{(j-m)!}\nonumber\\
& ~~~~~+
\sum_{j=1}^{k}2^{j}|A_{k-j}(x)|\int_{0}^{x}\frac{(x-s)^{j-1}}{(j-1)!}
|v^{(k)}(s)|^{2}ds.
\end{align}
On the other hand, similar to \eqref{60} we have
\begin{align}
u^{(k-j)}(x)=\sum_{m=1}^{j}u^{(k-m)}(0)\frac{x^{j-m}}{(j-m)!}+\int_{0}^{x}\frac{(x-s)^{j-1}}{(j-1)!}u^{(k)}(s)ds,\label{62}
\end{align}
and so
\begin{align}
2^{j}u^{(k-j)}(x)=2^{j}\sum_{m=1}^{j}u^{(k-m)}(0)\frac{x^{j-m}}{(j-m)!}+2^{j}\int_{0}^{x}\frac{(x-s)^{j-1}}{(j-1)!}u^{(k)}(s)ds.\label{63}
\end{align}
Furthermore, taking \eqref{63} into \eqref{59} yields
\begin{align} \label{64} u^{(k)}(x) &=\sum_{j=1}^{k}B_{k-j}(x)u^{(k-j)}(x)=\sum_{j=1}^{k}2^{-j}B_{k-j}(x)2^{j}u^{(k-j)}(x)
 \notag\\
& = \sum_{j=1}^{k}2^{-j}B_{k-j}(x)2^{j}\sum_{m=1}^{j}u^{(k-m)}(0)\frac{x^{j-m}}{(j-m)!}\nonumber\\
&
\quad+\sum_{j=1}^{k}2^{-j}B_{k-j}(x)2^{j}\int_{0}^{x}\frac{(x-s)^{j-1}}{(j-1)!}u^{(k)}(s)ds.
\end{align}
Then by \eqref{61} and \eqref{64} we obtain 
   \begr
& &|v^{(k)}(x)|^{2}-u^{(k)}(x)\nonumber\\
&
\leqslant&\sum_{j=1}^{k}|A_{k-j}(x)||v^{(k-j)}(x)|^{2}-\sum_{j=1}^{k}B_{k-j}(x)u^{(k-j)}(x)\notag\\
 &
=&\sum_{j=1}^{k}2^{j}\bigg(|A_{k-j}(x)|\sum_{m=1}^{j}|v^{(k-m)}(0)|^{2}\frac{x^{j-m}}
{(j-m)!}-2^{-j}B_{k-j}(x)\sum_{m=1}^{j}u^{(k-m)}(0)\frac{x^{j-m}}{(j-m)!}\bigg)\notag\\
& &
+\sum_{j=1}^{k}2^{j}\bigg(|A_{k-j}(x)|\int_{0}^{x}\frac{(x-s)^{j-1}}{(j-1)!}
|v^{(k)}(s)|^{2}ds-2^{-j}B_{k-j}(x)\int_{0}^{x}\frac{(x-s)^{j-1}}{(j-1)!}u^{(k)}(s)ds\bigg).\notag
 \endr
 Since $$ |A_{k-j}(x)|\leqslant 2^{-j}B_{k-j}(x) $$
 for all $ x \in
[a, 1)\backslash E $ and  $ j=1, 2, \cdots, k $, and $
|v^{(k-j)}(0)|^{2}\leqslant u^{(k-j)}(0) $ for all $ j=1, \cdots, k
$, we get 
 \begr
 &&|v^{(k)}(x)|^{2}-u^{(k)}(x)\nonumber\\
 & \leqslant&\sum_{j=1}^{k}|A_{k-j}(x)||v^{(k-j)}(x)|^{2}-\sum_{j=1}^{k}B_{k-j}(x)u^{(k-j)}(x)\notag\\
&\leqslant&
\sum_{j=1}^{k}B_{k-j}(x)\int_{0}^{x}\frac{(x-s)^{j-1}}{(j-1)!}(|v^{(k)}(s)|^{2}-u^{(k)}(s))ds\nonumber\\
&=&\int_{0}^{x}\bigg(\sum_{j=1}^{k}\frac{(x-s)^{j-1}B_{k-j}(x)}{(j-1)!}\bigg)(|v^{(k)}(s)|^{2}-u^{(k)}(s))ds\notag\\
&\coloneqq&\int_{0}^{x}R(x,
s)(|v^{(k)}(s)|^{2}-u^{(k)}(s))ds,\label{65}
\endr
where
 \begin{align*} R(x,
s)=\sum_{j=1}^{k}\frac{(x-s)^{j-1}B_{k-j}(x)}{(j-1)!}.\end{align*}
  By \eqref{65} and using the iteration we obtain 
\begr &&|v^{(k)}(x)|^{2}-u^{(k)}(x)\nonumber\\
& \leqslant&\int_{0}^{x}R(x,
s)(|v^{(k)}(s)|^{2}-u^{(k)}(s))ds\nonumber\\
&\leqslant& \int_{0}^{x}R(x, s)\int_{0}^{s}R(s,
s_{1})(|v^{(k)}(s_{1})|^{2}-u^{(k)}(s_{1}))ds_{1}ds\notag\\
&\leqslant& \cdots\nonumber\\
& \leqslant& \int_{0}^{x}R(x,
s)\int_{0}^{s}\cdots\int_{0}^{s_{m-1}}R(s_{m-1},
s_{m})(|v^{(k)}(s_{m})|^{2}-u^{(k)}(s_{m}))ds_{m}\cdots
ds_{1}ds\notag.
\endr
For a fixed point $ x\in[0, 1) $, it is easy to see that
\begin{align}
C(x)\coloneqq\int_{0}^{x}R^{m+1}(x,
s_{m})\big||v^{(k)}(s_{m})|^{2}-u^{(k)}(s_{m})\big|ds_{m} \notag
\end{align} is bounded about the fixed point $ x\in[0, 1) $, then by
$ B_{k-j}(x)\geqslant 0 $ for all $ x \in [0, 1)\backslash E $ and $
j=1, \cdots, k, $ and $ R(x, s)=0 $ for $ 0<x<s<1 $, we have

\begin{align}|v^{(k)}(x)|^{2}-u^{(k)}(x)\leqslant \frac{1}{m!}\int_{0}^{x}R^{m+1}(x,
s_{m})\big||v^{(k)}(s_{m})|^{2}-u^{(k)}(s_{m})\big|ds_{m}=\frac{C(x)}{m!}.\label{66}\end{align} Since $ C(x)/m!\rightarrow 0 $
as $ m\rightarrow +\infty $, we know the left side of \eqref{66} can be arbitrary small, and hence $ |v^{(k)}(x)|^{2}\leqslant u^{(k)}(x)
$ on $ [0, 1)\backslash E $. Furthermore by \eqref{57} and \eqref{63} we have
   \begr
&&|v^{(k-j)}(x)|^{2}-2^{j}u^{(k-j)}(x)\nonumber\\
&=&2^{j}\sum_{m=1}^{j}\frac{x^{j-m}}{(j-m)!}\bigg(|v^{(k-m)}(0)|^{2}-u^{(k-m)}(0)\bigg)\nonumber\\
&& +2^{j}\int_{0}^{x}\frac{(x-s)^{j-1}}{(j-1)!}  \bigg(|v^{(k)}(s)|^{2}-u^{(k)}(s)\bigg)ds\notag\\
&\leqslant& 2^{j}\int_{0}^{x}\frac{(x-s)^{j-1}}{(j-1)!} \bigg(|v^{(k)}(s)|^{2}-u^{(k)}(s)\bigg)ds.\label{67}
\endr
Applying $ |v^{(k)}(x)|^{2}\leqslant u^{(k)}(x) $ on $ [0,1)\backslash E $ to \eqref{67} we have
  $$ |v^{(j)}(x)|^{2}\leqslant 2^{k-j}u^{(j)}(x) $$ for all $ x \in [0, 1)\backslash E $ and  $
j=0, \cdots, k-1. $ Combining the estimates for $ j=k $ and $ j=0,
\cdots, k-1, $ the conclusion of Lemma \ref{69} follows, and the
proof is complete.
\end{proof}

Now, we generalize Theorem 5.1  in \cite{hkr1} to the case of
nonlinear differential equation, and we give the Theorem \ref{70}
below.  In fact, let us replace $ \nu<r<R $ in \cite[Theorem 5.1]{hkr1} with
$ \nu<r<1 $. Taking $ \varrho=(1+r)/2 $, then $ \nu<r<\varrho<1 $
and let $ \varepsilon_{0}>0 $. Thus
 $$h_{\theta}(x)\coloneqq
\max_{j=0,\cdots, k-1}n_{0}|A_{j}(xe^{i\theta})|^{1/(k-j)} $$ is
Riemann integrable on $ [\nu, \varrho] $. So there exists a
partition $ P=\{\nu=x_{0}, x_{1}, x_{2},\cdots, x_{n-1},
x_{n}=\varrho \} $ of $ [\nu, \varrho] $ such that $ x_{j}\neq r $
for all $ j=0,\cdots,n $, and
\begin{align*}
U(P,
h_{\theta})-\int_{\nu}^{\varrho}h_{\theta}(s)ds<\varepsilon_{0},
\end{align*}
where $ U(P, h_{\theta}) $ is the upper Riemann sum of $ h_{\theta}
$ corresponding to the partition $ P $. Let
\begin{align}&g_{\theta}(t)=n_{c}\cdot\sup_{x_{j}\leqslant x\leqslant
x_{j+1}}h_{\theta}(x),\quad x_{j}\leqslant x\leqslant x_{j+1},
j=0,\cdots,n-1, \label{79}\\& \delta_{j}\coloneqq \left\{
\begin{array}{ll}
 0, & \text{if $A_{j}(z)\equiv 0 $,}\\
1, & \text{otherwise},
\end{array}\right.\quad j=0,\cdots,k-1, \notag \\& C_{0}\coloneqq \max_{j=0,\cdots,
k-1}\bigg(\frac{|f^{(j)}(z_{\theta})|^{n_{0}}}{g_{\theta}(\nu)^{j}}\bigg)\nonumber
\shortintertext{and}
&(\upsilon^{(k)})^{n_{0}}+p_{k-1}(t)(\upsilon^{(k-1)})^{n_{0}}+\cdots+p_{1}(t)(\upsilon')^{n_{0}}+p_{0}(t)(\upsilon)^{n_{0}}=0,\notag
\end{align}
where $ n_{0}>1 $ and $ p_{j}(t)\coloneqq
e^{in_{0}(k-j)\theta}A_{j}(te^{i\theta}) $ for $ j=0,\cdots, k-1 $.
Then we get $$ |p_{j}(t)|=|A_{j}(te^{i\theta})|\leqslant
n_{0}^{-(k-j)}(1/n_{c})g_{\theta}(t)^{k-j}\delta_{j} .$$  Now let us
replace these new quantities and notations with the corresponding
ones used in \cite[Theorem 5.1]{hkr1}, and then by applying the same
reasoning used in \cite[Theorem 5.1]{hkr1} and Lemma \ref{69} above
we can obtain the following result.

\begin{theorem}\label{70}
Let $ f $  be a solution of \eqref{19} with $ n_{j}=n_{0}>1 $ for $
j=1,\cdots, k $, let $ n_{c} $ be the number of nonzero coefficients
$ A_{j}(z) $ for $ j=0, \cdots, k-1,$ and let $ \theta\in [0, 2\pi)
$ and $ \varepsilon>0 $. If $ z_{\theta}=\nu e^{i\theta}\in
\mathbb{D} $ satisfies $ A_{j}(z_{\theta})\neq 0 $ for some $
j=0,\cdots, k-1 $, then for all $ \nu<r<1 $,
\begin{align}\label{86} |f(re^{i\theta})|^{n_{0}}\leqslant C\exp
\bigg(n_{c}\int_{\nu}^{r}\max_{j=0,\cdots,
k-1}n_{0}|A_{j}(te^{i\theta})|^{1/(k-j)}dt\bigg),\end{align} where $
C>0 $ is a constant satisfying
\begin{align*} C\leqslant (1+\varepsilon)n_{0}^{k}\max_{j=0,\cdots, k-1}\bigg(\frac{|f^{(j)}(z_{\theta})|^{n_{0}}}{(n_{c})^{j}
\max_{n=0,\cdots, k-1}n^{
j}_{0}|A_{n}(z_{\theta})|^{j/(k-n)}}\bigg).
\end{align*}
\end{theorem}

By reviewing the reasoning of Theorem \ref{70} above, we  can obtain
the following result.

\begin{coro}\label{10} Let $ f $ be a solution of \eqref{19} with $ n_{j}=n_{0}>1 $ for
$j=1,\cdots, k $, let $ n_{c} $ be the number of nonzero
coefficients $ A_{j}(z) $ for $ j=0, \cdots,k-1,$ and let $
\theta\in [0, 2\pi) $  and $ \varepsilon>0 $. If $ z_{\theta}=\nu
e^{i\theta}\in \mathbb{D} $ satisfies $A_{j}(z_{\theta})\neq 0 $ for
some $ j=0,\cdots, k-1 $, then for $ r\in [\nu, 1) $ we have
\begin{align}\label{80}
 |f^{(j)}(re^{i\theta})|^{n_{0}}\leqslant&C\bigg(\sup_{\nu\leqslant x\leqslant(1+r)/2}\bigg(\max_{j=0,\cdots,
k-1}n_{0}|A_{j}(xe^{i\theta})|^{1/(k-j)}\bigg)\bigg)^{j}\nonumber\\
&\times\exp
\bigg(n_{c}\int_{\nu}^{r}\max_{j=0,\cdots,k-1}n_{0}|A_{j}(te^{i\theta})|^{1/(k-j)}dt\bigg)\end{align}
for $j=0,\cdots,k, $ and $ C>0 $ is a constant satisfying
\begin{align*}C\leqslant (1+\varepsilon)n_{c}^{j}n_{0}^{k-j}\max_{j=0,\cdots, k-1}\bigg(\frac{|f^{(j)}(z_{\theta})|^{n_{0}}}{(n_{c})^{j}
\max_{n=0,\cdots, k-1}n^{j}_{0}|A_{n}(z_{\theta})|^{j/(k-n)}}\bigg).
\end{align*}
\end{coro}

\begin{proof} For $ r\in [\nu,(1+r)/2]\backslash P $, it follows from the proof of Theorem \ref{70} that
\begin{align}\label{81}|f^{(j)}(re^{i\theta})|^{n_{0}}\leqslant C(g_{\theta}(r))^{j}\exp
\bigg(n_{c}\int_{\nu}^{r}\max_{j=0,\cdots,
k-1}n_{0}|A_{j}(te^{i\theta})|^{1/(k-j)}dt\bigg),\end{align} where $
P=\{\nu=x_{0}, x_{1}, x_{2},\cdots, x_{n-1}, x_{n}=(1+r)/2 \} $ is a
partition of $ [\nu, (1+r)/2] $ as defined in the proof of Theorem
\ref{70}, and \begin{align*} C\leqslant(1+\varepsilon)
n_{0}^{k-j}\max_{j=0,\cdots,
k-1}\bigg(\frac{|f^{(j)}(z_{\theta})|^{n_{0}}}{(n_{c})^{j}
\max_{n=0,\cdots, k-1}n^{j}_{0}|A_{n}(z_{\theta})|^{j/(k-n)}}\bigg).
\end{align*}
By \eqref{79} we have \begin{align} \label{82}
g_{\theta}(r)&=n_{c}\cdot\sup_{x_{j}\leqslant x\leqslant
x_{j+1}}h_{\theta}(x)\leqslant n_{c}\cdot\sup_{\nu\leqslant
x\leqslant(1+r)/2}h_{\theta}(x)\nonumber\\
&=n_{c}\cdot\sup_{\nu\leqslant
x\leqslant(1+r)/2}\bigg(\max_{j=0,\cdots,
k-1}n_{0}|A_{j}(xe^{i\theta})|^{1/(k-j)}\bigg),
\end{align}
  for $ x_{j}\leqslant r\leqslant x_{j+1} $, $ j=0,\cdots,n-1.$ Thus taking \eqref{82} into \eqref{81} we have
 \begin{align}|f^{(j)}(re^{i\theta})|^{n_{0}}\leqslant&
C\bigg(\sup_{\nu\leqslant x\leqslant(1+r)/2}\bigg(\max_{j=0,\cdots,
k-1}n_{0}|A_{j}(xe^{i\theta})|^{1/(k-j)}\bigg)\bigg)^{j}\nonumber\\
&\times \exp \bigg(n_{c}\int_{\nu}^{r}\max_{j=0,\cdots,
k-1}n_{0}|A_{j}(te^{i\theta})|^{1/(k-j)}dt\bigg)\end{align}
 for $
j=0,\cdots,k.$ Since $ r\in[\nu, 1) $ is arbitrary, the conclusion
follows and the proof is complete.
\end{proof}

By Theorem \ref{70},  we get the following result.

\begin{theorem}\label{87}  Suppose $ 0\leqslant s<1 $. Let all coefficients $ A_{j}(z) $ of
\eqref{19} with $ n_{j}=n_{0}>1 $ for $ j=1,\cdots, k $,   belong to the space $ H_{s}^{\infty} $, let $ n_{c} $ be the number of nonzero
coefficients $ A_{j}(z) $, $ j=0, \cdots, k-1 $ and let $ \theta\in [0, 2\pi) $ and $ \varepsilon>0 $. If $ z_{\theta}=\nu
e^{i\theta}\in \mathbb{D} $ satisfies $ A_{j}(z_{\theta})\neq 0 $ for some $ j=0,\cdots, k-1 $, then all solutions of \eqref{19}
belong to the $ H^{\infty} $ space.
\end{theorem}

\begin{proof} Since the conditions in Theorem \ref{87} satisfy the conditions of Theorem
\ref{70}, all solutions of \eqref{19} with $ n_{j}=n_{0}>1 $ for $j=1,\cdots, k $ have the estimate \eqref{86}. In addition, since all coefficients
of \eqref{19} belong to $ H_{s}^{\infty} $,   we have 
$$
|A_{j}(z)|\leqslant \|A_{j}\|_{H_{s}^{\infty}}/(1-|z|^{2})^{s},
~~\mbox{ for } j=0,\cdots,k-1 .$$
  Take $L=max_{j=0,\cdots,k-1}\{\|A_{j}\|_{H_{s}^{\infty}}\} $. From
\eqref{86} we have
\begin{align}|f(re^{i\theta})|^{n_{0}}&\leqslant C\exp
\bigg(n_{c}\int_{\nu}^{r}\max_{j=0,\cdots,
k-1}n_{0}|A_{j}(te^{i\theta})|^{1/(k-j)}dt\bigg)\notag\\&\leqslant
C\exp \bigg(n_{c}\int_{\nu}^{r}\max_{j=0,\cdots,
k-1}n_{0}L^{1/(k-j)}(1-t^{2})^{-s/(k-j)}dt\bigg).\notag\end{align}
Without loss of generality, taking $ L=1 $ we obtain
  \begin{align}
|f(re^{i\theta})|^{n_{0}}&\leqslant C\exp
\bigg(n_{c}\int_{\nu}^{r}\max_{j=0,\cdots,
k-1}n_{0}(1-t^{2})^{-s/(k-j)}dt\bigg)\nonumber\\
&\leqslant C\exp
\bigg(n_{c}n_{0}\int_{\nu}^{r}(1-t^{2})^{-s}dt\bigg)<\infty.\end{align}
Thus $ f\in H^{\infty} $, and the conclusion follows. The proof is
complete.
\end{proof}
By Corollary \ref{10} we have the following result.
\begin{theorem} \label{76} Suppose $ 0<s<\infty $. Let all the coefficients $ A_{j}(z) $ of
\eqref{19} with $ n_{j}=n_{0}>1 $ for $ j=1,\cdots, k $,   belong to
the Bloch space $ \mathfrak{B} $. Let $ \theta\in [0, 2\pi) $,   $
\varepsilon>0 $ and $ n_{c} $ be the number of nonzero coefficients
$ A_{j}(z) $, $ j=0, \cdots, k-1 $. If $z_{\theta}=\nu
e^{i\theta}\in \mathbb{D} $ satisfies $A_{j}(z_{\theta})\neq 0 $ for
some $ j=0,\cdots, k-1 $, then all solutions of \eqref{19} belong to
the space $ \cap_{ 0<s<\infty}\mathfrak{B}^{s} $.
\end{theorem}

\begin{proof} If the coefficients of  \eqref{19} with $ n_{j}=n_{0}>1 $ for $j=1,\cdots, k $,  belong to the Bloch space $ \mathfrak{B} $,  then
\eqref{80} gives that 
  \begin{align} \label{83}
|f'(re^{i\theta})|^{n_{0}}\leqslant &C\bigg(\sup_{\nu\leqslant x\leqslant(1+r)/2}\bigg(\max_{j=0,\cdots,
k-1}n_{0}|A_{j}(xe^{i\theta})|^{1/(k-j)}\bigg)\bigg)\nonumber\\
&\times \exp \bigg(n_{c}\int_{\nu}^{r}\max_{j=0,\cdots,k-1}n_{0}|A_{j}(te^{i\theta})|^{1/(k-j)}dt\bigg),
\end{align}
where $ C $ is a constant satisfying
\begin{align*}C\leqslant (1+\varepsilon)n_{c}n_{0}^{k-1}\max_{j=0,\cdots, k-1}\bigg(\frac{|f^{(j)}(z_{\theta})|^{n_{0}}}{(n_{c})^{j}
\max_{n=0,\cdots, k-1}n^{j}_{0}|A_{n}(z_{\theta})|^{j/(k-n)}}\bigg).
\end{align*}
Since $ A_{j}\in \mathfrak{B} $ for $ j=0, \cdots, k-1 $,   we have
$$ |A_{j}(z)|\leqslant (1/2)\|A_{j}\|_{\mathfrak{B}}\log((1+|z|)/(1-|z|)),~~ j=0, \cdots, k-1 .$$
 Set $ M\coloneqq max_{j=0, \cdots, k-1}\{\|A_{j}\|_{\mathfrak{B}}\} $. Then  \eqref{83} gives
\begin{align}
 |f'(re^{i\theta})|^{n_{0}}\leqslant&C\bigg(\sup_{\nu\leqslant x\leqslant(1+r)/2}\bigg(\max_{j=0,\cdots,
k-1}n_{0}|A_{j}(xe^{i\theta})|^{1/(k-j)}\bigg)\bigg)\nonumber\\
&\times \exp \bigg(n_{c}\int_{\nu}^{r}\max_{j=0,\cdots, k-1}n_{0}|A_{j}(te^{i\theta})|^{1/(k-j)}dt\bigg)\notag\\
\leqslant &C\bigg(\sup_{\nu\leqslant
x\leqslant(1+r)/2}\bigg(\max_{j=0,\cdots,
k-1}n_{0}\bigg(\frac{M}{2}\log\frac{1+x}{1-x}\bigg)^{1/(k-j)}\bigg)\bigg)\nonumber\\
&\times \exp \bigg(n_{c}\int_{\nu}^{r}\max_{j=0,\cdots,
k-1}n_{0}\bigg(\frac{M}{2}\log\frac{1+t}{1-t}\bigg)^{1/(k-j)}dt\bigg)\notag\\
\leqslant& C\bigg(\sup_{\nu\leqslant
x\leqslant(1+r)/2}\bigg(\max_{j=0,\cdots,
k-1}n_{0}\bigg(\frac{M}{2}\log\frac{3+r}{1-r}\bigg)^{1/(k-j)}\bigg)\bigg)\nonumber\\
&\times \exp \bigg(n_{c}\int_{\nu}^{r}\max_{j=0,\cdots,
k-1}n_{0}\bigg(\frac{M}{2}\log\frac{1+t}{1-t}\bigg)^{1/(k-j)}dt\bigg)\notag.
\end{align}
Without loss of generality, taking $ r\geqslant (e-1)/(e+1) $, denoted by $ c_{0} $, and $ M=2 $, we get
  \begr
 &&|f'(re^{i\theta})|^{n_{0}}\leqslant C\bigg(n_{0}\log\frac{3+r}{1-r}\bigg)\exp
\bigg(n_{c}n_{0}\int_{0}^{r}\max_{j=0,\cdots, k-1}\bigg(\log\frac{1+t}{1-t}\bigg)^{1/(k-j)}dt\bigg)\notag\\
&=&C\bigg(n_{0}\log\frac{3+r}{1-r}\bigg)\exp \bigg(n_{c}n_{0}\bigg(\bigg(\int_{0}^{c_{0}}+\int_{c_{0}}^{r}\bigg)\max_{j=0,\cdots,
k-1}\bigg(\log\frac{1+t}{1-t}\bigg)^{1/(k-j)}dt\bigg)\bigg)\notag\\
&=&C\bigg(n_{0}\log\frac{3+r}{1-r}\bigg)\exp \bigg(n_{c}n_{0}\bigg(\int_{0}^{c_{0}}\bigg(\log\frac{1+t}{1-t}\bigg)^{1/k}dt+\int_{c_{0}}^{r}
\log\frac{1+t}{1-t}dt\bigg)\bigg)\notag\\
&\leqslant& C\log(1/(1-r)).\label{84}
\endr
Taking any $ s\in (0, \infty) $, multiplying $(1-r^{2})^{n_{0}s} $ on both sides of \eqref{84}
 and taking $\sup_{re^{i\theta}\in \mathbb{D}} $, we obtain
\begin{align}
\sup_{re^{i\theta}\in \mathbb{D}}|f'(re^{i\theta})|^{n_{0}}(1-r^{2})^{n_{0}s}\leqslant
C\sup_{re^{i\theta}\in \mathbb{D}}(1-r^{2})^{n_{0}s}\log(1/(1-r))<\infty,\notag\end{align}
which implies that $ f\in \mathfrak{B}^{s} $. Since $ s\in (0,\infty) $ is arbitrary, the conclusion follows.
\end{proof}

Finally, we give another estimate for solutions of \eqref{19} with $n_{k}\geqslant n_{j} $ for $ j=0, 1, \cdots, k-1 $ and $
n_{k}\geqslant 1 $ and these conditions are weaker than that in Theorem \ref{70}.

\begin{theorem} Suppose the positive real number $ n_{k}\geqslant n_{j} $ for $ j=0,
1, \cdots, k-1 $ and $ n_{k}\geqslant 1 $ in \eqref{19}. Let $ f $ be a solution of \eqref{19}. Then for $ z=re^{i\theta}\in \mathbb{D}
$, we have
\begin{align}
|f(re^{i\theta})|\leqslant \sum_{m=1}^{k}|f^{(k-m)}(0)|\frac{r^{k-m}}{(k-m)!}+\int_{0}^{r}\frac{(r-s)^{k-1}}{(k-1)!}\bigg(\sum_{i=0}^{\infty}\mathcal
{H}_{i}(s)\bigg)^{1/n_{k}}ds,\notag
\end{align}
where $ \sum_{i=0}^{\infty}\mathcal {H}_{i}(s) $ is a continuous function on $ [0, 1) $, dependent on $ k $, $ A_{k-j} $ for $
j=1,\cdots,k $, $ n_{j} $ for $ j=0, \cdots, k, $ and the initial values $ f^{(j)}(0) $ for $ j=0,\cdots,k-1 $, and the path of
integration is a line segment in $ \mathbb{D} $ joining origin $ 0 $ and $ z $.
\end{theorem}

\begin{proof}
Let us integrate $ f^{(k)} $ for $ j $ times from origin $ 0 $ to $z $ along the line segment in $ \mathbb{D} $ joining $ 0 $ and $ z
$, and direct computation gives
\begin{align}
|f^{(k-j)}(re^{i\theta})|\leqslant \sum_{m=1}^{j}|f^{(k-m)}(0)|\frac{r^{j-m}}{(j-m)!}+\int_{0}^{r}\frac{(r-s)^{j-1}}{(j-1)!}|f^{(k)}(se^{i\theta})|ds.
\label{1}
\end{align}
 Then  
\begin{align}|f^{(k-j)}(re^{i\theta})|^{n_{k-j}}\leqslant&
n_{k-j}^{j}\sum_{m=1}^{j}|f^{(k-m)}(0)|^{n_{k-j}}\bigg(\frac{r^{j-m}}{(j-m)!}\bigg)^{n_{k-j}}\nonumber\\
&+n_{k-j}^{j}\bigg(\int_{0}^{r}\frac{(r-s)^{j-1}}{(j-1)!}
|f^{(k)}(se^{i\theta})|ds\bigg)^{n_{k-j}}\label{2}
\end{align}
for $ j=0, 1, \cdots, k $. By \eqref{19} we have
  $$
|f^{(k)}|^{n_{k}}\leqslant \sum_{j=1}^{k}|A_{k-j}(re^{i\theta})|
|f^{(k-j)}|^{n_{k-j}},$$
 which, together with \eqref{2}, yields
\begin{align}|f^{(k)}(re^{i\theta})|^{n_{k}}\leqslant&
\sum_{j=1}^{k}|A_{k-j}(re^{i\theta})||f^{(k-j)}|^{n_{k-j}}\notag
\\\leqslant &\sum_{j=1}^{k}n_{k-j}^{j}|A_{k-j}(re^{i\theta})|
\sum_{m=1}^{j}|f^{(k-m)}(0)|^{n_{k-j}}\bigg(\frac{r^{j-m}}{(j-m)!}\bigg)^{n_{k-j}}\nonumber\\
&+\sum_{j=1}^{k}n_{k-j}^{j}|A_{k-j}(re^{i\theta})|\bigg(\int_{0}^{r}\frac{(r-s)^{j-1}}{(j-1)!}
|f^{(k)}(se^{i\theta})|ds\bigg)^{n_{k-j}}.\notag \end{align}
 It is
well known that for $ a\geqslant 0, b\geqslant 0 $, we have $
ab\leqslant a^{p}/p+b^{q}/q $ for $ 1\leqslant p\leqslant \infty,
1\leqslant q\leqslant \infty $ and $ 1/p+1/q=1 $. Taking $p=n_{k}/n_{k-j}\geqslant 1 $, and $ q=n_{k}/(n_{k}-n_{k-j}) $ we
have 
 \begr
&&n_{k-j}^{j}|A_{k-j}(re^{i\theta})|\bigg(\int_{0}^{r}\frac{(r-s)^{j-1}}{(j-1)!}
|f^{(k)}(se^{i\theta})|ds\bigg)^{n_{k-j}}\nonumber\\
&\leqslant&
\frac{n_{k}-n_{k-j}}{n_{k}}(n_{k-j}^{j}|A_{k-j}(re^{i\theta})|)^{n_{k}/(n_{k}-n_{k-j})}
+\frac{n_{k-j}}{n_{k}}\bigg(\int_{0}^{r}\frac{(r-s)^{j-1}}{(j-1)!}
|f^{(k)}(se^{i\theta})|ds\bigg)^{n_{k}},\nonumber
\endr
and so 
  \begin{align}|f^{(k)}(re^{i\theta})|^{n_{k}}\leqslant&
\sum_{j=1}^{k}|A_{k-j}(re^{i\theta})||f^{(k-j)}|^{n_{k-j}}\notag \\
\leqslant&\sum_{j=1}^{k}n_{k-j}^{j}|A_{k-j}(re^{i\theta})|
\sum_{m=1}^{j}|f^{(k-m)}(0)|^{n_{k-j}}\bigg(\frac{r^{j-m}}{(j-m)!}\bigg)^{n_{k-j}}\nonumber\\
&+
\sum_{j=1}^{k}\frac{n_{k}-n_{k-j}}{n_{k}}(n_{k-j}^{j}|A_{k-j}(re^{i\theta})|)^{n_{k}/(n_{k}-n_{k-j})}\notag\\
&+\sum_{j=1}^{k}
\frac{n_{k-j}}{n_{k}}\bigg(\int_{0}^{r}\frac{(r-s)^{j-1}}{(j-1)!}
|f^{(k)}(se^{i\theta})|ds \bigg)^{n_{k}}.\notag\end{align}
  By
H\"{o}lder inequality we have \begin{align}
|f^{(k)}(re^{i\theta})|^{n_{k}}\leqslant&
\sum_{j=1}^{k}\bigg(n_{k-j}^{j}|A_{k-j}(re^{i\theta})|
\sum_{m=1}^{j}|f^{(k-m)}(0)|^{n_{k-j}}\bigg(\frac{r^{j-m}}{(j-m)!}\bigg)^{n_{k-j}}\nonumber\\
&+
\frac{n_{k}-n_{k-j}}{n_{k}}(n_{k-j}^{j}|A_{k-j}(re^{i\theta})|)^{n_{k}/(n_{k}-n_{k-j})}\bigg)\notag\\
&+\sum_{j=1}^{k}
\frac{n_{k-j}}{n_{k}}r^{n_{k}-1}\int_{0}^{r}\frac{(r-s)^{n_{k}(j-1)}}{((j-1)!)^{n_{k}}}
|f^{(k)}(se^{i\theta})|^{n_{k}}ds\notag\\
\coloneqq& H(r)+\int_{0}^{r}L(r,
s)|f^{(k)}(se^{i\theta})|^{n_{k}}ds,\label{11}
\end{align}
where \begin{align}
H(r)=&\sum_{j=1}^{k}\bigg(n_{k-j}^{j}|A_{k-j}(re^{i\theta})|
\sum_{m=1}^{j}|f^{(k-m)}(0)|^{n_{k-j}}\bigg(\frac{r^{j-m}}{(j-m)!}\bigg)^{n_{k-j}}\nonumber\\
&+\frac{n_{k}-n_{k-j}}{n_{k}}(n_{k-j}^{j}|A_{k-j}(re^{i\theta})|)^{n_{k}/(n_{k}-n_{k-j})}\bigg)\nonumber
\end{align}
and
 \begin{align} L(r, s)=\sum_{j=1}^{k}
\frac{n_{k-j}}{n_{k}}\frac{(r-s)^{n_{k}(j-1)}r^{n_{k}-1}}{((j-1)!)^{n_{k}}}.\nonumber\end{align}
By \eqref{11} we have
\begin{align}
|f^{(k)}(re^{i\theta})|^{n_{k}}&\leqslant H(r)+\int_{0}^{r}L(r,
s)|f^{(k)}(se^{i\theta})|^{n_{k}}ds\notag\\&\leqslant
H(r)+\int_{0}^{r}L(r, s)H(s)ds  \notag\\
 & ~~~~+ \int_{0}^{r}L(r, s)\int_{0}^{s}L(s,
s_{1})|f^{(k)}(s_{1}e^{i\theta})|^{n_{k}}ds_{1}ds\notag\\
 &\coloneqq
\mathcal {H}_{0}(r)+\mathcal {H}_{1}(r)+\mathcal {L}_{1}(r),
\end{align}
where $ \mathcal {H}_{0}(r)=H(r) $, $\mathcal
{H}_{1}(r)=\int_{0}^{r}L(r, s)H(s)ds$ and
\begin{align*}  \mathcal {L}_{1}(r)=\int_{0}^{r}L(r,
s)\int_{0}^{s}L(s, s_{1})|f^{(k)}(s_{1}e^{i\theta})|^{n_{k}}ds_{1}ds
\end{align*} are the results of the first iteration by \eqref{11}, and
the $ n $-th iteration  gives
\begin{align}
|f^{(k)}(re^{i\theta})|^{n_{k}}\leqslant\sum_{i=0}^{n}\mathcal
{H}_{i}(r)+\mathcal {L}_{n}(r).
\end{align}
For a fixed point $ r\in[0, 1) $, denote \begin{align*}
S(r)\coloneqq \max_{0\leqslant s\leqslant r}H(s),\,\,T(r)\coloneqq
\max_{0\leqslant r_{1}, r_{2}\leqslant r}L(r_{1},
r_{2})\,\,\text{and}\,\, M(r)\coloneqq \max_{0\leqslant
r_{1}\leqslant r}
 |f^{(k)}(r_{1}e^{i\theta})|^{n_{k}}. \end{align*}
Then for $n\geqslant 3 $, by induction we have
\begin{align}\mathcal{H}_{n}(r)&=\int_{0}^{r}L(r, s)\int_{0}^{s}L(s,
s_{1})\int_{0}^{s_{1}}\cdots\int_{0}^{s_{n-2}}L(s_{n-2},
s_{n-1})H(s_{n-1})ds_{n-1}\cdots
ds\nonumber\\
&\leqslant\frac{T^{n}(r)S(r)}{n!},\notag
\end{align}
and so $ \Sigma_{i=0}^{n}\mathcal {H}_{i}(r) $ is uniformly
convergent to $ \Sigma_{i=0}^{\infty}\mathcal {H}_{i}(r) $ on any
close set $ [0, r] $. Hence $ \Sigma_{i=0}^{\infty}\mathcal
{H}_{i}(r) $ is a continuous function on $ [0, 1) $ and so $
\Sigma_{i=0}^{\infty}\mathcal {H}_{i}(r) $ is integrable on $ [0,1)
$. For $ n\geqslant 1 $, by induction we get
\begin{align} \mathcal{L}_{n}(r)&=\int_{0}^{r}L(r,
s)\int_{0}^{s}L(s,s_{1})\int_{0}^{s_{1}}\cdots\int_{0}^{s_{n-1}}L(s_{n-1},
s_{n})|f^{(k)}(s_{n}e^{i\theta})|^{n_{k}}ds_{n}\cdots ds\nonumber\\
&\leqslant\frac{T^{n+1}(r)M(r)}{n!}\rightarrow
0,\,\,\text{as}\,\,n\rightarrow \infty. \notag
\end{align}
Henceforth,
\begin{align}\label{12}
|f^{(k)}(re^{i\theta})|^{n_{k}}\leqslant\sum_{i=0}^{\infty}\mathcal
{H}_{i}(r).
\end{align}
Taking \eqref{12} to \eqref{1} we obtain
\begin{align}
|f^{(k-j)}(re^{i\theta})|&\leqslant
\sum_{m=1}^{j}|f^{(k-m)}(0)|\frac{r^{j-m}}{(j-m)!}+\int_{0}^{r}\frac{(r-s)^{j-1}}{(j-1)!}|f^{(k)}(se^{i\theta})|ds\notag\\&\leqslant
\sum_{m=1}^{j}|f^{(k-m)}(0)|\frac{r^{j-m}}{(j-m)!}+\int_{0}^{r}\frac{(r-s)^{j-1}}{(j-1)!}\bigg(\sum_{i=0}^{\infty}\mathcal
{H}_{i}(r)\bigg)^{1/n_{k}}ds, \notag
\end{align}
for $ j=0, 1, \cdots, k $. Therefore by taking $ j=k $ we have
\begin{align}
|f(re^{i\theta})|\leqslant
\sum_{m=1}^{k}|f^{(k-m)}(0)|\frac{r^{k-m}}{(k-m)!}+\int_{0}^{r}\frac{(r-s)^{k-1}}{(k-1)!}\bigg(\sum_{i=0}^{\infty}\mathcal
{H}_{i}(s)\bigg)^{1/n_{k}}ds. \notag
\end{align}
 The proof is complete.
\end{proof}

\noindent{\bf Acknowledgement:} The first author is supported by
NNSF of China (No. 11126284) and the
scholarship awarded by China Scho1arship Council (No. [2014]7325).
The second author is supported by NNSF of China (No. 11471143).


\begin{thebibliography}{99999}
{\footnotesize 
\bibitem{be} D. Benbourenane and L. Sons, On global solutions of complex differential equations in the unit disk, {\it Complex Var.} {\bf 49} (2004), 913--925.

\bibitem{cy} T. Cao and H. Yi, The growth of solutions of linear differential equations with coefficients of iterated order in
the unit disc, {\it J. Math. Anal. Appl.} {\bf 319} (2006),
278--294.


\bibitem{cs} Z. Chen and K. Shon, The growth of solutions of differential equations with coefficients of small growth in the
disc, {\it J. Math. Anal. Appl. } {\bf 297} (2004), 285--304.

\bibitem{cgh} I. Chyzhykov, G. Gundersen and J. Heittokangas, Linear differential equations and logarithmic derivative estimates,
{\it Proc. London Math. Soc.} {\bf 86} (2003), 735--754.


\bibitem{gsw} G. Gundersen, E. Steinbart and S. Wang, The possible orders of solutions of linear differential equations with
polynomial coefficients, {\it Trans. Amer. Math. Soc.} {\bf 350}
(1985), 1225--1247.

\bibitem{he} J. Heittokangas, On complex differential equations in the unit disc, {\it Ann. Acad. Sci. Fenn. Math. Diss.} {\bf 122}
(2000), 1--54.

\bibitem{hkr1} J. Heittokangas, R. Korhonen and J. R\"{a}tty\"{a}, Growth estimates for solutions of linear complex differential
equations, {\it Ann. Acad. Sci. Fenn.} {\bf 29} (2004), 233--246.

\bibitem{hkr2} J. Heittokangas, R. Korhonen and J. R\"{a}tty\"{a}, Fast growing solutions of linear differential equations in the unit
disc, {\it Results Math.} {\bf 49} (2006), 265--278.

\bibitem{hkr3} J. Heittokangas, R. Korhonen and J. R\"{a}tty\"{a}, Linear differential equations with solutions in the Dirichlet type
subspace of the Hardy space, {\it Nagoya Math. J.} {\bf 187} (2007),
91--113.

\bibitem{hkr4} J. Heittokangas, R. Korhonen and J. R\"{a}tty\"{a}, Linear differential equations with coefficients in weighted Bergman
and Hardy spaces, {\it Trans. Amer. Math. Soc.} {\bf 360} (2008),
1035--1055.

\bibitem{her} H. Herold,  Ein Vergleichssatz f\"{u}r komplexe lineare Differentialgleichungen, {\it Math. Z.} {\bf 126} (1972), 91--94.

\bibitem{hi} E. Hille, {\it Ordinary Differential Equations in the Complex Domain}, Pure and Applied Mathematics, Wiley-Interscience
(John Wiley \& Sons), New York-London-sydney, 1976.

\bibitem{kr} R. Korhonen and J. R\"{a}tty\"{a}, Linear differential equations in the unit disc with analytic solutions of finite order,
{\it Proc. Amer. Math. Soc.} {\bf 135} (2007), 1355-1363.

\bibitem{kr1} R. Korhonen and J. R\"{a}tty\"{a}, Finite order solutions of linear differential equations in the unit disc,
{\it J. Math. Anal. Appl.} {\bf 349} (2009), 43--54.

\bibitem{la} I. Lain, {\it Nevanlinna Theory and Complex Differential Equations}, Walter de Gruyter, Berlin, 1993.

\bibitem{hw} H. Li and H. Wulan, Linear differential equations with solutions in $ Q_{K} $ spaces, {\it J. Math. Anal. Appl.} {\bf 375} (2011),
478--489.

\bibitem{po} Ch. Pommerenke, On the mean growth of the solutions of complex linear differential equations in the disk, {\it Complex Var. } {\bf 1}
(1982),  23--38.

\bibitem{r} J. R\"{a}tty\"{a}, Linear differential equations with solutions in Hardy spaces, {\it Complex Var. Elliptic Equ.} {\bf 52} (2007), 785--795.


\bibitem{ww} H. Wulan and P. Wu, Characterizations of $ Q_{T} $ spaces, {\it J. Math. Anal. Appl.} {\bf 254} (2001),  484--497.

\bibitem{wuzhu} H. Wulan and K. Zhu, $ Q_{K} $ spaces via higher order derivatives, {\it Rocky Mountain J. Math.} {\bf 38} (2008),
329--350.

\bibitem{zhu2} K. Zhu, Bloch type spaces of analytic functions, {\it Rocky Mountain J. Math.} {\bf 23} (1993), 1143--1177.

}
\end{thebibliography}
\end{document}